\documentclass{article}

\usepackage{amsmath, amsthm, amssymb}
\usepackage[pdftex]{graphicx}
\newtheorem{theorem}{Theorem}
\newtheorem{definition}{Definition}

\author{Koki Suetsugu}
\title{Discovering a new universal partizan ruleset}

\begin{document}
\maketitle

\begin{abstract}
    In Combinatorial Game Theory, we study the set of games   $\mathbb{G}$, whose elements are mapped from positions of rulesets.
In many case, given a ruleset, not all elements of $\mathbb{G}$ can be given as a position in the ruleset. It is an intriguing question what kind of ruleset would allow all of them to appear. In this paper, we introduce a ruleset named {\sc turning tiles} and prove the ruleset is a universal partizan ruleset, that is, every element in $\mathbb{G}$ can occur as a position in the ruleset. This is the second universal partizan ruleset after {\sc generalized konane.}
\end{abstract}

\section{Introduction}
Combinatorial Game Theory studies algebraic structures of two-player games with no chance move nor hidden information.
Here, we introduce some definitions and theorems of Combinatorial Game Theory for later discussion. For more details of Combinatorial Game Theory, see \cite{LIP, CGT}.

In this theory, the two players are called Left and Right.
The term "game" is polysemous, and we refer to each position as a {\it game}. The description of what moves are allowed for a given position is called the {\it ruleset}.

A game is defined by Left and Right options recursively.
\begin{definition}
$\{ \mid \}$ is a game.
Let $G^L_1, G^L_2, \ldots, G^L_n, G^R_1, G^R_2, \ldots, G^R_m$ be games. $\{ G^L_1, G^L_2, \ldots, G^L_n \mid G^R_1, G^R_2, \ldots, G^R_m\}$ is also a game. $G^L_1, G^L_2, \ldots, G^L_n$ are called left options and $G^R_1, G^R_2, \ldots, G^R_m$ are called right options.

Let $\mathbb{G}$ be the set of all games.
\end{definition}
As this definition, in every game, the sets of left and right options do not need to be the same. Such games are called {\em partizan} games.  

The position with neither option, $\{\mid\}$, is called the {\it terminal position}, or $0$. In this paper, we assume that the play is under {\em normal play} convention, that is, the player who moves last is the winner. 

We denote by $\mathcal{L}$, $\mathcal{R}$, $\mathcal{N}$, and $\mathcal{P}$ the set of positions in which $\mathcal{L}$eft, $\mathcal{R}$ight, the $\mathcal{N}$ext player, and the $\mathcal{P}$revious player have winning strategies, respectively. Every position belongs to exactly one of the sets. We say the sets are outcomes of the games. For a game $G,$ let $o(G)$ be the outcome of $G$.
We define the partial order of outcomes as $\mathcal{L} > \mathcal{P} > \mathcal{R}, \mathcal{L} > \mathcal{N} > \mathcal{R}.$

The {\it disjunctive sum} of games is an important concept in Combinatorial Game Theory. For games $G$ and $H$, a position in which a player makes a move for one or the other on their turn is called a disjunctive sum of $G$ and $H$, or $G + H$. More precisely, it is as follows:

\begin{definition}
If the game trees of $G$ and $H$ are isomorphic, then we say these games are isomorphic or $G \cong H.$
\end{definition}

\begin{definition}
For games $G \cong \{G^L_1, G^L_2 \ldots G^L_n \mid G^R_1, G^R_2, \ldots, G^R_{m}\}$ and $H \cong \{ H^L_1, H^L_2 \ldots, H^L_{n'} \mid H^R_1, H^R_2, \ldots, H^R_{m'}\},
G+H \cong \{G + H^L_1, G+H^L_2, \ldots, G+H^L_{n'}, G^L_1 + H, G^L_2+H, \ldots, G^L_n + H \mid G+H^R_1, G+H^R_2, \ldots, G+H^R_{m'}, G^R_1+H, G^R_2+H, \ldots, G^R_m + H \}.$
\end{definition}

We also define equality, inequality and negative of games.
\begin{definition}
If for any $X, o(G + X)$ is the same as $o(H + X),$ then we say $G = H$. 
\end{definition}

\begin{definition}
If $o(G + H) \geq o(H + X)$ holds for any $X,$ then we say $G \geq H.$
On the other hand, if $o(G + H) \leq o(H + X)$ holds for any $X,$ then we say $G \leq H.$
We also say $G \gtrless H$ if $ G \not \geq H$ and $G \not \leq H.$
\end{definition}

\begin{definition}
For a game $G \cong  \{G^L_1, G^L_2, \ldots, G^L_n \mid G^R_1, G^R_2, \ldots, G^R_m\},$ let $-G \cong  \{ -G^R_1, -G^R_2, -G^R_m \mid -G^L_1, -G^L_2, \ldots, -G^L_n\}.$

$G + (-H)$ is denoted by $G - H$. 
\end{definition}

It is known that $(\mathbb{G}, +, =)$ is an aelian group and $(\mathbb{G}, \geq, =)$ is a partial order. That is, for any games $G, H$ and $J$, following conditions are satisfied;

\begin{itemize}
\item $(G + H) + J = G + (H + J)$.
\item $G + 0 = G$.
\item $G - G = 0$.
\item $G + H = H + G$.
\item $G \leq G.$
\item $ G \leq H, H \leq J \Rightarrow G \leq J.$ 
\item If $G \leq H$ and $H \leq G,$ then $G = H.$ 
\end{itemize}

It is also known that $G = 0 \Leftrightarrow o(G) = \mathcal{P}.$
From these results, it is clear that $G = H \Leftrightarrow G-H=0$. This means that if we need to find out whether two games $G$ and $H$ are equal, we can check whether the previous player has a winning strategy in $G-H$. This fact is often used.

The question arises here, will there be a ruleset in which for any game there is a position equal to the game? If the games appearing in each ruleset are restricted, then perhaps we should think in a narrower framework.
In fact, however, it is known that for every game, a position equal to the game appears in the ruleset called {\sc generalized konane}.








\subsection{The rule of {\sc generalized konane}}
{\sc Generalized konane} is a ruleset based on {\sc konane,} a traditional Hawaiian ruleset.
The ruleset of {\sc generalized konane} is as follows:
\begin{itemize}
    \item There are black and white pieces on a grid-like board. Left moves the black pieces and Right moves the white pieces.
    \item On their turn, a player can jump over an adjacent opponent's piece by their piece to eliminate it. However, the square at the end of the jump must be empty.
    \item If a piece is adjacent to another opponent's piece after a jump, the player can continue to jump over it if it is in the same direction.
    \item The player who moves last is the winner. 
\end{itemize}

Figure \ref{figkonane} is a position of {\sc generalized konane}. 
\begin{center}
\begin{figure}[tb]
\label{figkonane}
\includegraphics[height=8cm]{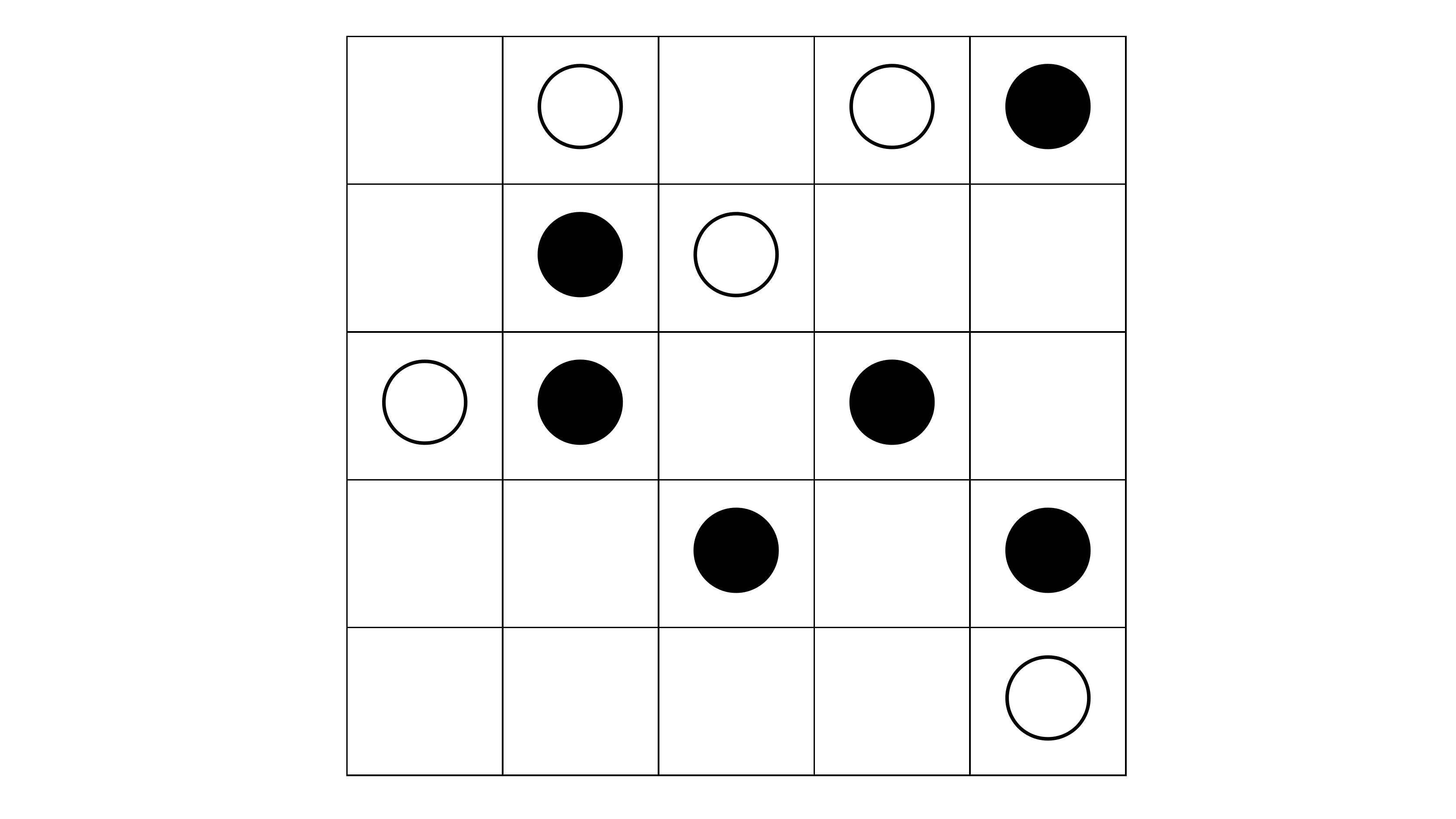}
\caption{Position in {\sc generalized konane}.}
\end{figure}
\end{center}


A {\em universal partizan ruleset} is a ruleset such that every game in $\mathbb{G}$ is equal to a position of the ruleset.
{\sc Generalized konane} was proved to be a universal partizan ruleset by Carvalho and Santos in \cite{CS} and this ruleset has been the only ruleset proved as a universal partizan ruleset. Thus, to find a new universal partizan ruleset is an interesting problem and it was introduced in \cite{Now} as a unsolved problems in Combinatorial Game Theory.
In this paper, we introduce a ruleset and prove the ruleset is also a universal partizan ruleset.

\subsection{The rule of {\sc turning tiles}}
In this paper, we introduce a ruleset named {\sc turning tiles.}
The ruleset of {\sc turning tiles} is as follows:
\begin{itemize}
    \item Square tiles are laid out. The front side is red or blue, and the back side is black. 
    \item Some pieces are on tiles.
    \item Each player (Left, whose color is bLue and Right, whose color is Red), in their turn, take a piece and move the piece straight on the tiles of their color.
    \item Tiles on which the piece pass over are turned over.
    \item The player who moves last is the winner.    
\end{itemize}

\begin{center}
\begin{figure}[tb]
\includegraphics[height=12cm]{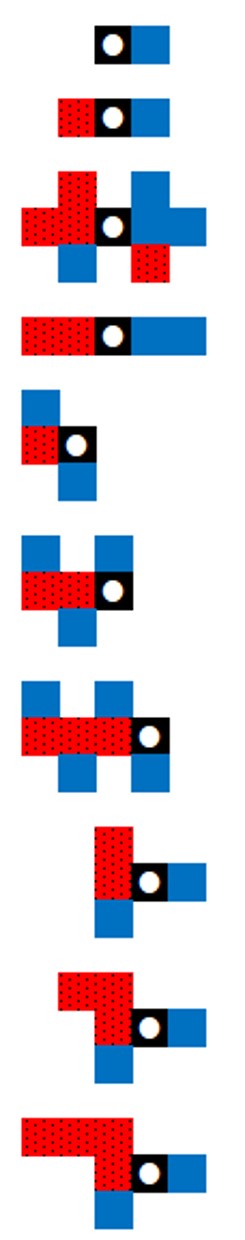}
\caption{Some simple values in {\sc turning tiles} \\From top to bottom: $1,*,*2,\pm 1,
\label{simplevalues1}
\frac{1}{2}, \frac{1}{4}, \frac{1}{8},\uparrow,$ tiny$1$, tiny$2$}
\end{figure}
\end{center}
Figure \ref{simplevalues1} shows some positions of {\sc turning tiles}. For the meanings of values, see the textbooks of Combinatorial Game Theory like \cite{CGT, LIP}.

As far as the author knows, this ruleset has not been invented before, but it seems simple, natural, and fun to play. The main result of this paper is that this ruleset is universal partizan ruleset.

The paper is organized as follows. In Section~\ref{sec2}, we prove that {\sc turning tiles} is a universal partizan ruleset. The final section presents the conclusions.

\section{Main result}
\label{sec2}
In this section, we prove that {\sc turning tiles} is a universal partizan ruleset.
\begin{theorem}
{\sc Turning tiles} is a universal partizan ruleset even if the number of pieces is one.
\end{theorem}
\begin{proof}
We prove this theorem by a similar way to the proof of the theorem that {\sc generalized konane} is a universal partizan ruleset shown in \cite{CS}.

We show every $G \cong \{ G^L_1, \ldots, G^L_k \mid G^R_1, \ldots, G^R_l  \}$ is constructed as Figure \ref{Rec}.

\begin{center}
\begin{figure}[tb]
\includegraphics[height=10cm]{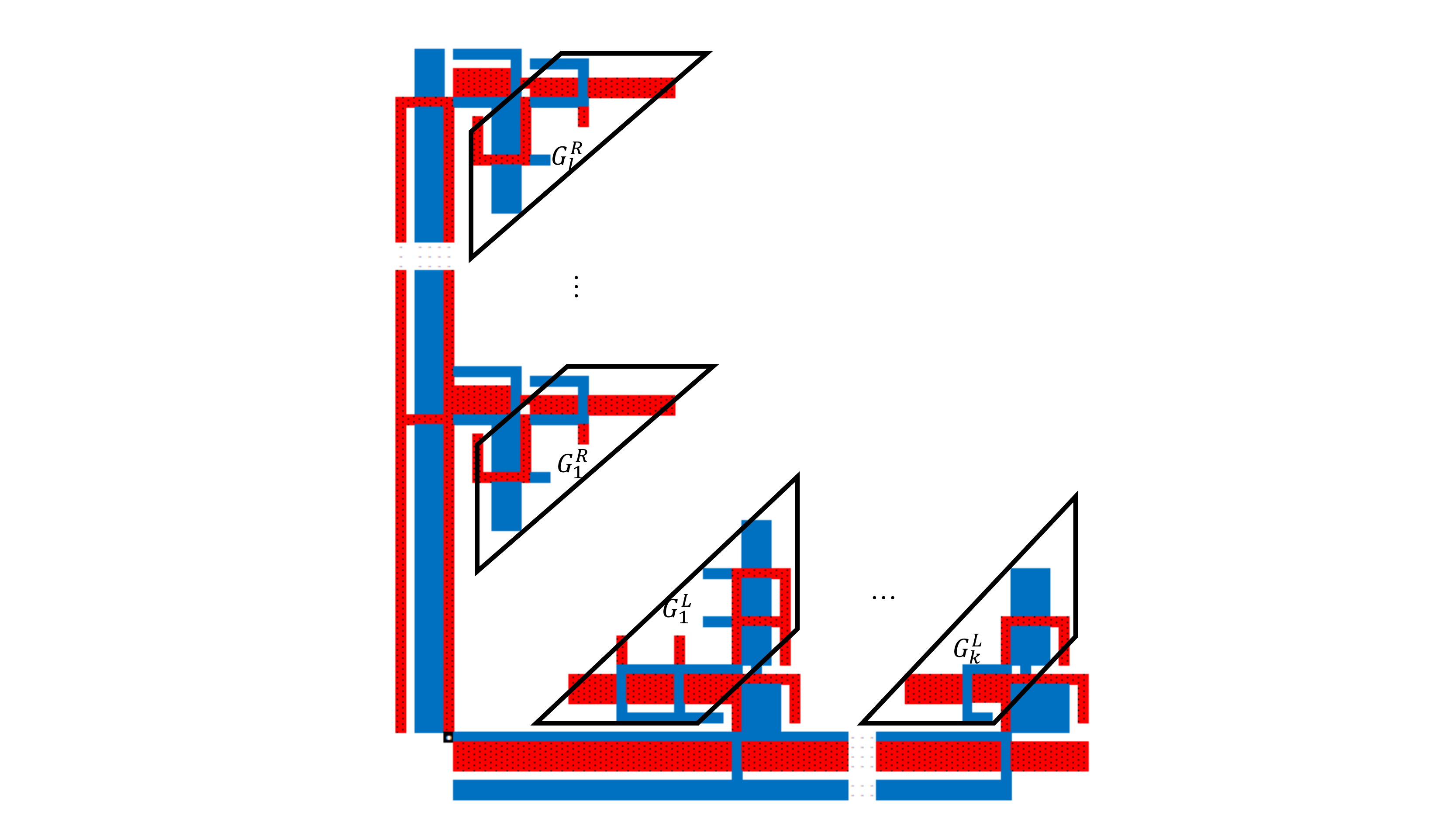}
\caption{Recursive composition of $G = \{G^L_1, \ldots, G^L_k \mid G^R_1, \ldots, G^R_l\}$.}
\label{Rec}
\end{figure}
\end{center}

We prove this theorem by induction. Assume that each $G^L_1,  \ldots, G^L_k, G^R_1,  \ldots, G^R_l$ has already constructed in the same way.

We say some points are first, second and third connection points of options and some parts are connection arms as Figures \ref{term1} and \ref{term2}. Assume that connection arms are enough long.

\begin{center}
\begin{figure}[tb]
\includegraphics[height=10cm]{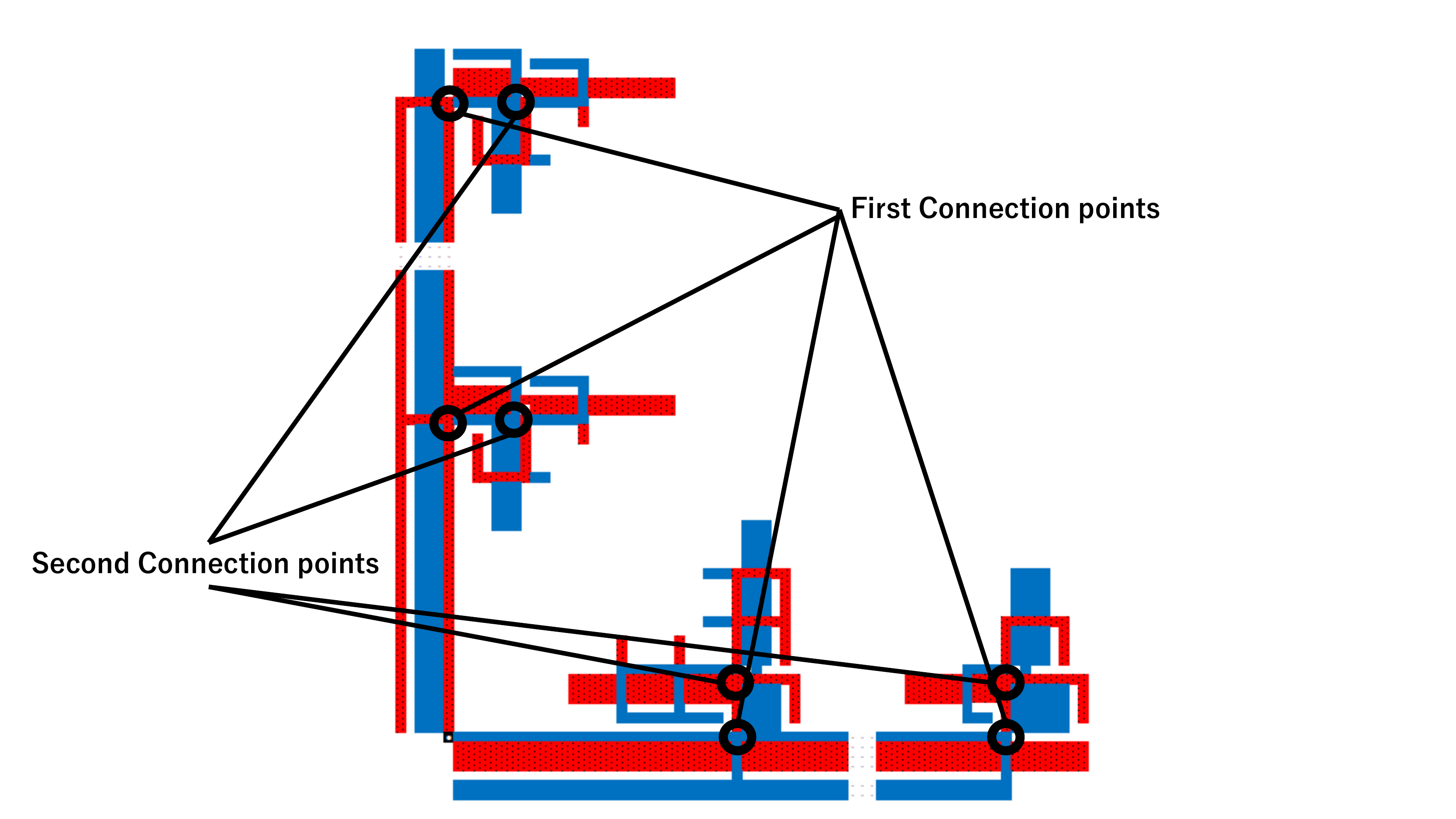}
\caption{First and second connection points.}
\label{term1}
\end{figure}
\end{center}
\begin{center}
\begin{figure}[tb]
\includegraphics[height=10cm]{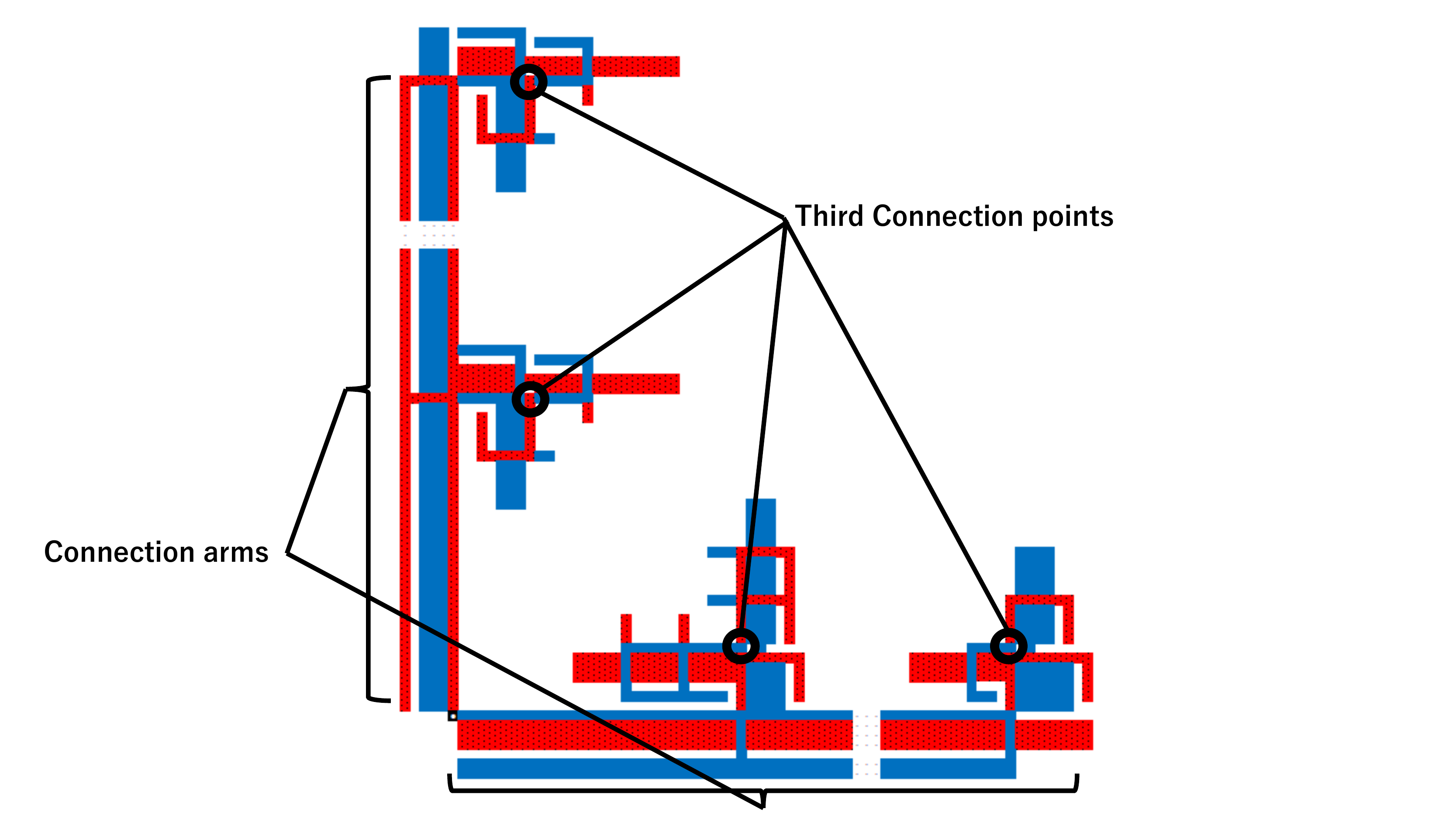}
\caption{Third connection points and connection arms.}
\label{term2}
\end{figure}
\end{center}

Let this position be $X.$ We prove that $o(X - G) = \mathcal{P}.$

Suppose that Left moves first. Consider the case Left moves a piece on $X$. Except for the case to move the piece one of the first connection points, Right can win because he can move the piece to enough wide red area. Thus, without loss of generality, we assume that Left moves the piece to the first connection point  of $G^L_1$. 

Immediately after the move, Right moves the piece to second connection point of $G^L_1$ and force Left to move third connection point of $G^L_1$ (because otherwise Right can win by moving the piece to enough wide red area), and if Left moves so, then the whole position changes to $G^L_1 - G.$ Therefore, Right can move to $G^L_1 -G^L_1 = 0$ and he wins.

Consider the other case, that is, Left moves $X - G$ to $X - G^R_i.$
Without loss of generality, we assume that $i = 1.$  Then Right moves the piece to first connection point of $G^R_1.$ Left has to move the piece to second connection point of $G^R_1$ because otherwise Right can win by moving the piece to enough wide red area. However, even if Left moves the piece to second connection point of $G^R_1,$ Right can moves the piece to third connection point of $G^R_1$ and then the whole position becomes $G^R_1 - G^R_1 = 0$ and Right wins.

Therefore, for $X-G,$ Right wins if Left moves first. By using similar way, Left wins if Right moves first. Thus, $X = G.$

\end{proof}

\section{Conclusion}
\label{sec3}
In this paper, we have defined a new ruleset, {\sc turning tiles}, and have proven that it is a universal partizan ruleset. This result gives a solution to an open problem in combinatorial game theory. The ruleset is simple and not unnatural even if we consider actual play. The fact that we were able to show that even such a ruleset is a universal partizan ruleset suggests that {\sc generalized konane} is not a very special case and there may be more universal partizan rulesets in the commonly played rulesets.

\end{document}